\documentclass[11pt,a4paper]{article}
\usepackage{amsmath,amssymb,amsthm,txfonts,amsbsy}
\usepackage{authblk}
\usepackage{graphicx}
\usepackage{color}
\usepackage{comment}
\usepackage[graph,color]{xy}
\usepackage{ulem}

\definecolor{bcP}{rgb}{1,0,1}

\usepackage{tikz, framed}
\usetikzlibrary{arrows,shapes,matrix,decorations.pathmorphing}
\usetikzlibrary{backgrounds}

\theoremstyle{plain} 
\newtheorem{Theorem}{Theorem}

\newtheorem{Lemma}[Theorem]{Lemma}
\newtheorem{Proposition}[Theorem]{Proposition}
\newtheorem{Corollary}[Theorem]{Corollary}

\theoremstyle{definition}
\newtheorem{Definition}[Theorem]{Definition}

\theoremstyle{remark}

\newtheorem{Remark}[Theorem]{Remark}

\newcommand{\A}{{\mathcal{A}}}
\newcommand{\kS}{{\mathcal{S}}}
\newcommand{\I}{{\kI}}

\newcommand{\kB}{{\mathcal{B}}}

\newcommand{\kC}{\mathcal{C}}
\newcommand{\kL}{\mathcal{L}}
\newcommand{\kI}{\mathcal{I}}

\title{Alternatives for the $q$-matroid axioms of independent spaces, bases, and spanning spaces}
\author[1]{Michela Ceria}
\affil[1]{\small 
Dept. of Mechanics, Mathematics \& Management, Politecnico di Bari,  Via Orabona 4 - 70125 Bari - Italy; michela.ceria@gmail.com}

\author[2]{Relinde Jurrius}
\affil[2]{\small Faculty of Military Sciences, Netherlands Defence Academy, The Netherlands; rpmj.jurrius@mindef.nl}
 
\begin{document}

\maketitle
\begin{abstract}
It is well known that in $q$-matroids, axioms for independent spaces, bases, and spanning spaces differ from the classical case of matroids, since the straightforward $q$-analogue of the classical axioms does not give a $q$-matroid. For this reason, a fourth axiom has been proposed. In this paper we show how we can describe these spaces with only three axioms, providing two alternative ways to do that. As an application, we show direct cryptomorphisms between independent spaces and circuits and between independent spaces and bases. \textcolor{red}{This version contains corrections to the published version.}
\end{abstract}

\section{Introduction}

The study of $q$-matroids originates in \cite{crapo1964theory}, but has been re-discovered in \cite{JP18} because of its link to network coding. Many notions in network coding are $q$-analogues of notions associated to error-correcting codes in the Hamming metric. Generally speaking, a $q$-analogue in combinatorics is a generalisation from a finite set to a finite dimensional vector space. The $q$-analogue of codes in the Hamming metric are codes in the rank metric. The $q$-analogue of a combinatorial design, called a subspace design, is a special case of a code in the subspace metric. Both types of codes are of interest for network coding. An overview of the foundational work in this area can be found in \cite{greferath2018network}.

Just as matroids are related to codes and designs, their $q$-analogues are related as well. As mentioned, $q$-(poly)matroids are a generalisation of rank-metric codes \cite{GLJ,gorla2019rank,JP18,shiromoto19}. Furthermore, $q$-matroids can be used to find new (weighted) subspace designs \cite{BCIR21,WINEpaper1}. These applications motivate the study of $q$-matroids from a theoretical point of view. \\

In the classical case of matroids independent sets need three axioms. For $q$-matroids, the straightforward $q$-analogue of these three axioms are not strong enough to get a $q$-matroid with a semimodular rank function. Therefore, a fourth axiom was added (see \cite{JP18}). \\
This unexpected fourth axiom raises some questions. Why would we need this extra axiom for independent spaces, bases, and spanning spaces, but not for things like dependent spaces, circuits, flats, and hyperplanes? Can't we find a better way to describe the axioms for independent spaces, bases, and spanning spaces, using only three axioms? \\
In this paper we present a positive answer to this question. We propose two ways to define independent spaces, bases, and spanning spaces with only three axioms.  The first one is to remove the third axiom, because it is implied by the fourth one (where we have to be a bit careful for independent spaces to pick the right variation of the fourth axiom). The second one is an alternative for the third axiom that is still a $q$-analogue of the classical case, but that obliterates the need for the fourth axiom. \\
As an application of this restriction of the number of axioms, we prove two cryptomorphisms. The first one is a direct cryptomorphism between independent sets and circuits that was not shown before. The second one is a cryptomorphism between independent spaces and bases. This was done in \cite{JP18}, but we believe there was a gap in the proof that we will fix here.

\section{Preliminaries}

Throughout this paper, $n$ denotes a fixed positive integer and $E$ a fixed $n$-dimensional vector space over an arbitrary field $\mathbb{F}$. The notation $\mathcal{L}(E)$ indicates the \emph{lattice of subspaces} of $E$. For any $A,B\in\mathcal{L}(E)$ with $A\subseteq B$ we denote by $[A,B]$ the interval between $A$ and $B$, that is, the lattice of all subspaces $X$ with $A\subseteq X\subseteq B$. For $A\subseteq E$ we use the notation $\mathcal{L}(A)$ to denote the interval $[\{0\},A]$. For more background on lattices, see for example Birkhoff \cite{birkhoff}. \\

We use the following definition of a $q$-matroid.

\begin{Definition}\label{rankfunction}
A $q$-matroid $M$ is a pair $(E,r)$ where $r$ is an integer-valued  function defined on the subspaces of $E$ with the following properties:
\begin{itemize}
\item[(R1)] For every subspace $A\in \kL(E)$, $0\leq r(A) \leq \dim A$. 
\item[(R2)] For all subspaces $A\subseteq B \in \kL(E)$, $r(A)\leq r(B)$. 
\item[(R3)] For all $A,B\in \kL(E)$, $r(A+ B)+r(A\cap B)\leq r(A)+r(B)$.  
\end{itemize}
The function $r$ is called the \emph{rank function} of the $q$-matroid and the vector space $E$ is called the \emph{ground space} of the $q$-matroid. 
\end{Definition}

In order to make notation more compact, we use the following ways to describe families of subspaces.
 
\begin{Definition}\label{def:families}
Let $\A \subseteq \kL(E)$. We define the following family of subspaces of $E$:
   \begin{align*}
       \max(\A)&:=\{ X \in \A : X \nsubseteq A \text{ for any } A \in \A, A \neq X \}, \\ 
       \min(\A)&:=\{ X \in \A :  A \nsubseteq X \text{ for any } A \in \A, A \neq X \}.
   \end{align*}
For any subspace $X \in \kL(E)$, we define then the collection of \emph{maximal subspaces
	of $X$ in $\A$} to be the collection of subspaces
	\[
	\max(X,\A):=\{ A \in \A : A \subseteq X \text{ and } B \subset X, B \in \A \implies \dim(B) \leq \dim(A) \}.
	\]
	In other words, $\max(X,\A)$ is the set of subspaces of $X$ in $\A$ that have maximal dimension over all such choices of subspaces. Similarly, we define the \emph{minimal subspaces containing $X$ in $\A$} to be the collection of subspaces
	\[
	\min(X,\A):=\{ A \in \A : X \subseteq A \text{ and } X \subset B, B \in \A \implies \dim(B) \geq \dim(A) \}.
	\]
Finally, by slight abuse of notation, we write
\[ X\cap\A := \{X\cap A:A\in\A\}. \]
\end{Definition}

We define several specific subspaces in a $q$-matroid.

\begin{Definition}\label{independentspaces}
Let $M=(E,r)$ be a $q$-matroid. 
A subspace $A$ of $E$ is called an \emph{independent} space of $M$ if $r(A)=\dim A$. A subspace that is not an independent space is called a \emph{dependent space}. A minimal dependent space (w.r.t. inclusion) is called a \emph{circuit}. A \emph{spanning space} of $M$ is a subspace $S$ such that $r(S)=r(E)$. A \emph{loop} of $M$ is a $1$-dimensional subspace $\ell\subseteq E$ such that $r(\ell)=0$.
\end{Definition}

A $q$-matroid can be equivalently defined by its independent spaces, bases, spanning spaces and circuits. See \cite{bcj} for an overview of these cryptomorphic definitions and many others.

\begin{Definition}\label{independence-axioms}
	Let $\I \subseteq \mathcal{L}(E)$. We define the following \emph{independence axioms}.
	\begin{itemize}
		\item[(I1)] $\I\neq\emptyset$.
		\item[(I2)] For all $I,J \in \kL(E)$, if $J\in\I$ and $I\subseteq J$, then $I\in\I$.
		\item[(I3)] For all $I,J\in\I$ satisfying $\dim I<\dim J$, there exists a $1$-dimensional subspace $x\subseteq J$, $x\not\subseteq I$ such that $I+x\in\I$.
		\item[(I4)] For all $A,B \in \kL(E)$ and $I,J \in \kL(E)$ such that 
		$I \in \max(\I \cap \kL(A))$ and $J \in \max(\I \cap \kL(B))$,
		there exists $K\in \max(\I \cap \kL(A+B))$ such that $K \subseteq I+J$.
	\end{itemize}
	If $\I$ satisfies the independence axioms (I1)-(I4) we say that $(E,\I)$ is a collection of \emph{independent spaces}.
\end{Definition}

\begin{Definition}\label{indep-bases}
Let $\mathcal{B} \subseteq \kL(E)$.
We define the following \emph{basis axioms}.
\begin{itemize}
\item[(B1)] $\mathcal{B}\neq\emptyset$
\item[(B2)] For all $B_1,B_2\in\mathcal{B}$, if $B_1\subseteq B_2$, then $B_1=B_2$.
\item[(B3)] For all $B_1,B_2\in\mathcal{B}$ and for every subspace $A$ of codimension 1 in $B_1$ satisfying $B_1\cap B_2\subseteq A$, there is a $1$-dimensional subspace $y$ of $B_2$ such that $A+y\in\mathcal{B}$.
\item[(B4)] 
For all $A,B\in \kL(E)$ and $I,J\in\kL(E)$ such that $I\in\max(E,A\cap\kB)$ and $J\in\max(E,B\cap\kB)$, there exists $K\in\max(E,(A+B)\cap\kB)$ such that $K\subseteq I+J$.
\end{itemize}
If $\mathcal{B}$ satisfies the bases axioms (B1)-(B4) we say that $(E,\mathcal{B})$ is a collection of \emph{bases}.
\end{Definition}

\begin{Definition}\label{spanning-axioms}
	Let $\kS \subseteq \mathcal{L}(E)$. We define the following \emph{spanning space axioms}.
	\begin{itemize}
		\item[(S1)] $E \in \kS$.
		\item[(S2)] For all $I,J \in \kL(E)$, if $J\in\kS$ and $J \subseteq I$, then  $I\in\kS$.
		\item[(S3)] For all $I,J\in\kS$ such that $\dim J<\dim I$, there exists some $X \in \kL(E)$ of codimension $1$ in $E$ satisfying $J \subseteq X$, $I \nsubseteq X$, and $I\cap X\in\kS$.
		\item[(S4)] For all $A,B\in\kL(E)$ and $I,J\in \kL(E)$ such that $I\in\min(\kS\cap[A,E])$ and $J\in\min(\kS\cap[B,E])$, there exists $K\in\min(\kS\cap [A\cap B,E])$ such that $I\cap J\subseteq K$.		
	\end{itemize}
	If $\kS$ satisfies the independence axioms (S1)-(S4) we say that $(E,\kS)$ is a collection of \emph{spanning spaces}.
\end{Definition}

\begin{Definition}\label{circuit-axioms}
Let $\mathcal{C}\subseteq\mathcal{L}(E)$. We
define the following \emph{circuit axioms}.
\begin{itemize}
\item[(C1)] $\{0\}\notin\mathcal{C}$.
\item[(C2)] For all $C_1,C_2\in\mathcal{C}$, if $C_1\subseteq C_2$  $C_1=C_2$.
\item[(C3)]  For distinct $C_1,C_2 \in \kC$ and any $X\in \kL(E)$ of codimension $1$ there is a circuit $C_3 \in \kC$ such that $C_3 \subseteq (C_1+C_2)\cap X$.
\end{itemize}
If $\kC$ satisfies the circuit axioms (C1)-(C3), we say that $(E,\mathcal{C})$ is a collection of \emph{circuits}.
\end{Definition}
Note that the axiom (C3) listed here is different from the axiom (C3) as defined in \cite[Theorem 64]{JP18}. An explanation of this can be found in \cite[Section 11]{bcj}.

A lattice isomorphism between a pair of lattices $(\kL_1,\leq_1,\vee_1,\wedge_1)$ and $(\kL_2,\leq_2,\vee_2,\wedge_2)$ is a bijective function $\varphi:\kL_1\longrightarrow\kL_2$ that is order-preserving and preserves the meet and join, that is, for all $x,y\in\kL_1$ we have that $\varphi(x\wedge_1 y)=\varphi(x)\wedge_2\varphi(y)$ and $\varphi(x\vee_1 y)=\varphi(x)\vee_2\varphi(y)$. A lattice anti-isomorphism between a pair of lattices is a bijective function $\psi:\kL_1\longrightarrow\kL_2$ that is order-reversing and interchanges the meet and join, that is, for all $x,y\in\kL_1$ we have that $\psi(x\wedge_1 y)=\psi(x)\vee_2\psi(y)$ and $\psi(x\vee_1 y)=\psi(x)\wedge_2\psi(y)$.  
We hence define a notion of equivalence  and duality between $q$-matroids. 

\begin{Definition}
Let $E_1,E_2$ be vector spaces over the same field $\mathbb{F}$. Let $M_1=(E_1,r_1)$ and $M_2=(E_2,r_2)$ be $q$-matroids. We say that $M_1$ and $M_2$ are \emph{lattice-equivalent} or \emph{isomorphic} if there exists a lattice isomorphism $\varphi:\kL(E_1)\longrightarrow \kL(E_2)$ such that $r_1(A)=r_2(\varphi(A))$ for all $A\subseteq E_1$. In this case we write $M_1 \cong M_2$.
\end{Definition}
 
Fix an anti-isomorphism $\perp:\kL(E)\longrightarrow\kL(E)$ that is an involution. For any subspace $X \in \kL(E)$ we denote by $X^\perp$ the \emph{dual} of $X$ in $E$ with respect to $\perp$.

Important operations on $q$-matroids are restriction, contraction and duality. We give a short summary here and refer to \cite{BCIR21,JP18} for details.

\begin{Definition}\label{defdual}
Let $M=(E,r)$ be a $q$-matroid. Then $M^*=(E,r^*)$ is also a $q$-matroid, called the \emph{dual $q$-matroid}, with rank function
\[ r^*(A)=\dim(A)-r(E)+r(A^\perp). \]
\end{Definition}

\begin{Theorem}[\cite{JP18}, Theorem 45]
The subspace $B$ is a basis of $M$ if and only if $B^\perp$ is a basis of $M^*$.
\end{Theorem}

\begin{Definition}\label{restr}
Let $M=(E,r)$ be a $q$-matroid. The \emph{restriction} of $M$ to a subspace $X$ is the $q$-matroid $M|_X$ with ground space $X$ and rank function $r_{M|_X}(A)=r_M(A)$. \\
The \emph{contraction} of $M$ of a subspace $X$ is the $q$-matroid $M/X$ with ground space $E/X$ and rank function $r_{M/X}(A/X)=r_M(A)-r_M(X)$.
\end{Definition}

\section{Redundancy in the axiom systems}\label{redundancy}

In this section we show that (I3), (B3) and (S3) are implied by the other axioms. In case of (I4) this is actually a bit subtle, since the exact statement of (I4) has a somewhat vague history.

\subsection{A discussion on variations of (I4)}\label{discussionI4}

The axiom (I4) was first stated in \cite{JP18}. It was formulated using the ambiguous term ``maximal independent space inside $A$'' for some $A\subseteq E$. It was not clarified if this maximality was taken with respect to inclusion or dimension. However, if one carefully reads the proofs in \cite{JP18}, especially Proposition 15, it becomes clear that maximality is taken with respect to dimension. Intuitively this also follows from the fact that for the cryptomorphism between independence and rank the following rank function in terms of independence is defined:
\[ r_\kI(A)=\max\{\dim I:I\in\kI,I\subseteq A\}. \]
In following papers, notably \cite{bcj}, the ambiguity in (I4) was solved by assuming maximality was taken with respect to inclusion. This did not lead to any problems, since by (I3), both notions are equivalent. It is only in cases where (I3) is not assumed (or proven) that the difference in maximality matters.

In this section we discuss the relations between the following variations of the axiom (I4).

\begin{itemize}
\item[(oI4)]  For all $A,B \in \kL(E)$ and $I,J \in \kL(E)$ such that $I \in \max(A, \kI )$ and $J \in \max(B, \kI)$, there exists $K\in \max(A+B, \kI)$ such that $K \subseteq I+J$.
\item[(I4)] For all $A,B \in \kL(E)$ and $I,J \in \kL(E)$ such that $I \in \max(\I \cap \kL(A))$ and $J \in \max(\I \cap \kL(B))$, there exists $K\in \max(\I \cap \kL(A+B))$ such that $K \subseteq I+J$.
\item[(I4')] Let $A \in \kL(E)$ and let $I \in \max(A,\I)$. Let $B \in \kL(E)$. 
Then there exists $J \in \max(A+B,\I)$ such that  $J \subseteq I+B$.
\item[(I4'')] Let $A \in \kL(E)$ and let $I \in \max(A,\I)$. 
Let $x \in \kL(E)$ be a $1$-dimensional space. Then there exists $J \in \max(x+A,\I)$ such that $J \subseteq x+I$.
\end{itemize}

We use the notation (oI4) for the version of (I4) as implied in \cite{JP18} and we reserve (I4) for the version that appears in other papers. The following relation among these alternatives of (I4) were proven.

\begin{Theorem}[Theorem 26 of \cite{bcj}]\label{oldI4s}
Let $\kI$ be a collection of subspaces satisfying (I1)-(I3).
Then the axiom systems (I1)-(I4), (I1)-(I4') and (I1)-(I4'') are pairwise equivalent.
\end{Theorem}

The next result is a variation of this theorem, involving (oI4) instead of (I4) and not depending on (I3). In fact, this proof is already implicit in the proof of \cite[Theorem 26]{bcj}. This result will be of use in the next section.

\begin{Proposition}\label{newI4s}
Let $\kI$ be a collection of subspaces satisfying (I1) and (I2). Then the axiom systems (I1), (I2), (oI4); (I1), (I2), (I4'); and (I1), (I2), (I4'') are pairwise equivalent.
\end{Proposition}
\begin{proof}
In \cite[Theorem 26]{bcj} it is proven that $\kI$ satisfies (I4') if and only if it satisfies (I4''): the axiom (I3) is not used  there. It is straightforward that if $\kI$ satisfies (oI4) then it satisfies (I4'). For the implication in the other direction, exactly the same proof as in \cite[Theorem 26]{bcj} holds and this does not use (I3). In fact, what is proven there is that if $\kI$ satisfies (I4'), then it satisfies (oI4).
\end{proof}

\begin{Remark}\label{I4isDifferent}
By the previous discussion, it is now clear that the axiom (I4) is weaker with respect to the other variations.
Indeed, if we observe (oI4), as well as (I4') and (I4''), they actually don't need (I3) in proving their equivalence.
Things change if we want to prove their equivalence with (I4): in that proof (I3) becomes crucial. This will be made clearer in the next section.
\end{Remark}

\subsection{Redundancy of (I3)}

In this section we prove that the axiom (I3) is redundant, provided we use any variant of the fourth independence axiom that is not (I4). We do this by showing that given (I1) and (I2), the axioms (I3) and (I4) together are equivalent to the axiom (I4'') (or (I4') or (oI4)).

\begin{Theorem}
Let $(E,\kI)$ be a $q$-matroid. Then, for the set $\kI$, the axiom (oI4) holds true.
\end{Theorem}
\begin{proof}
This is a direct consequence of Proposition \ref{newI4s}.
\end{proof}

\begin{Theorem}\label{I3isImplByI4}
Let $E$ be a vector space and let $\tilde{\kI}$ be a collection of subspaces satisfying the axioms (I1), (I2) and (oI4). Then $(E,\tilde{\kI})$ is a $q$-matroid.
\end{Theorem}
\begin{proof}
We have to show that $\tilde{\kI}$ satisfies the axioms (I1), (I2), (I3) and (I4). The first two axioms are satisfied by definition and (I4) follows from Proposition \ref{newI4s} and Theorem \ref{oldI4s}. So it is left to prove (I3). \\
Let $I,J\in \tilde{\kI}$ with $\dim I<\dim J$. Assume, towards a contradiction, that for all $1$-dimensional spaces $x\subseteq J$, $x\not\subseteq I$ we have that $I+x\notin\tilde{\kI}$. Let $x_i$ be $1$-dimensional spaces in $J$ such that we can write $I+J=I\oplus x_1\oplus\cdots\oplus x_h$. Note that, by (I2), $x_i\in\tilde{\kI}$ for all $i$. \\
By Proposition \ref{newI4s}, $\tilde{\kI}$ satisfies (I4''). We apply (I4'') to $I$ and $x_1$: there is a maximal member of $\tilde{\kI}$ (w.r.t. dimension) contained in $I+x_1$. By assumption, $I+x_i\notin \tilde{\kI}$, so $I$ is such a maximal member of $\tilde{\kI}$ in $I+x_1$. Next, we apply (I4'') to $I+x_1$ and $x_2$: there is a maximal member of $\tilde{\kI}$ in $I+x_1+x_2$ contained in $I+x_2$. Again by assumption, $I+x_2\notin \tilde{\kI}$ so $I$ is such a maximal member of $\tilde{\kI}$. Continuing like this, we find that $I$ is a maximal member of $\tilde{\kI}$ in $I+x_1+\cdots+x_h=I+J$. However, this is a contradiction, since $J\subseteq I+J$, $J \in \tilde{\kI}$, and $\dim I<\dim J$. We conclude that (I3) has to hold and thus $(E,\tilde{\kI})$ is a $q$-matroid.
\end{proof}

The proof of (I3) is similar to Proposition 6 of \cite{JP18}, only written in terms of independence instead of rank. Also, note that the last part of this proof does not hold with (I4) instead of (I4''). This supports Remark \ref{I4isDifferent}.

\subsection{Redundancy of (B3) and (S3)}

For the basis axioms we have a similar result. Even though it is not specified, it is implied that in (B4) a maximal intersection of a space with a basis is an intersection of maximal dimension. Therefore, the subtleties we had with (I4) and (oI4) distinguishing between maximal w.r.t. inclusion and w.r.t. dimension, do not appear for bases.

\begin{Theorem}\label{B3isImplByB4}
Let $\kB$ be a family of subspaces of $E$ that satisfies the axioms (B1), (B2) and (B4). Then $\kB$ satisfies (B3).
\end{Theorem}
\begin{proof}
Let $B_1,B_2\in \kB$ and let $A\subseteq B_1$ be a codimension $1$ subspace such that $B_1\cap B_2\subseteq A$. Assume, towards a contradiction, that for all $1$-dimensional spaces $x\subseteq B_2$ we have that $A+x\notin\kB$. Since $B_1\cap B_2\subseteq A$, we will never have that $A+x=B_1$.
Let $x_i$ be $1$-dimensional spaces in $B_2$ such that we can write $A+B_2=A\oplus x_1\oplus\cdots\oplus x_h$. Note that $\max(E,x_i\cap\kB)=\{x_i\}$, since $x_i\subseteq B_2$.

Now apply (B4) to $A$ and $x_1$: there is a $J_1\in\max(E,(A+x_1)\cap\kB)$ such that $J_1\subseteq A+x_1$.
Since by assumption, $A+x_1\notin\kB$, we can take $J_1=A$.
Next, we apply (B4) to $A+x_1$ and $x_2$: there is a $J_2\in\max(E,(A+x_1+x_2)\cap\kB)$ such that $J_2\subseteq A+x_2$. Again by assumption, $A+x_2\notin\kB$ so we can take $J_2=A$. Continuing like this, we find that $A\in\max(E,(A+x_1+\cdots+x_h)\cap\kB)=\max(E,(A+B_2)\cap\kB)$. However, this is a contradiction, since $B_2\subseteq A+B_2$ and $\dim A<\dim B_2$. We conclude that (B3) has to hold.
\end{proof}

We finish with the result for spanning spaces, that can be proven by taking the dual arguments to the proofs for independent spaces.

\begin{Theorem}\label{S3isImplByS4}
Let $\kS$ be a family of subspaces of $E$. Define the following spanning axiom.
\begin{itemize}
\item[(oS4)] For all $A,B\in\kL(E)$ and $I,J\in \kL(E)$ such that $I\in\min(A,\kS )$ and $J\in\min(B, \kS)$, there exists $K\in\min(A\cap B, \kS)$ such that $I\cap J\subseteq K$.	
\end{itemize}
If $\kS$ satisfies the axioms (S1), (S2) and (oS4), then $\kS$ satisfies (S3).
\end{Theorem}

\section{An alternative for the axiom (I3)}

In this section, we propose a new version for the axiom (I3),  that we will call (nI3) and we will prove that it subsumes both the (I3) and  the (I4) axioms for a $q$-matroid. 

\subsection{Motivation}

Before we state the axiom (nI3), we will give some motivation for this statement. Let us look at a small example. Let $E=\mathbb{F}_2^2$ and let $x$, $y$ and $z$ be the three $1$-dimension spaces of $E$. If we let $\kI=\{x,\{0\}\}$, we have a family satisfying the axioms (I1), (I2) and (I3), but not (I4). The latter can be seen by applying (I4) to $y$ and $z$: they are both not in $\kI$, so a maximal member of $\kI$ should be inside $\{0\}$. However, this is a contradiction because $x\subseteq y+z$ and $x\in\kI$. \\
Define a rank function for all $A\subseteq E$ as $r(A)=\dim  (  \max\{I\subseteq A:I\in\kI\} 
) $. Then the rank function in our example is not semimodular, i.e., does not satisfy axiom (R3):
\[ r(y+z)+r(y\cap z)=r(E)+r(0)=1+0>r(y)+r(z)=0. \]
We want the rank function to be that of a $q$-matroid. How can we achieve this with little change to $\kI$? Note that we can still ask $z\notin\kI$ if then we let $y\in\kI$. This gives a mixed diamond, as explained in Appendix A.3 of \cite{ceriajurrius2021}.

\begin{Lemma}\label{LoopsInSameSubs}
If a $q$-matroid $M$ has loops, then they are exactly all $1$-dimensional subspaces of a subspace $L\subseteq E$.
\end{Lemma}
\begin{proof}
This is a direct consequence of Lemma 11 in \cite{JP18}, that says that if $x$ and $y$ are loops, then $x+y$ has rank $0$. Applying this iteratively, we find that the sum of any number of loops has rank $0$. Then axiom (r2) implies that all $1$-dimensional subspaces of this sum of loops have rank $0$, hence are loops themselves.
\end{proof}

\begin{Definition}\label{Loopspace}
The subspace of $E$ containing all loops is called the \emph{loop space} of $M$. We usually denote it by $L$.
\end{Definition}

What goes wrong in our small example has to do with the loop space. If you have one $1$-dimensional space that is independent, then all other $1$-dimensional spaces that are not in the loop space $L$ have to be independent. This applies to other dimensions as well, by applying contraction. \\
Let $I,J\in\kI$ with $\dim I<\dim J$. Then (I3) tells us that there is a $1$-dimensional space $x\subseteq J$, $x\not\subseteq I$ such that $I+x\in\kI$. If we consider $M/I$, we find that $(I+x)/I$ is a $1$-dimensional independent space in $M/I$. So, outside the loop space $L$ of $M/I$, all $1$-dimensional spaces in $M/I$ have to be independent. Since not every $1$-dimensional space in $M/I$ is a loop, the space $L$ has at least codimension $1$ in $E/I$. \\
Now we will translate this to $M$. The independent $1$-dimensional spaces $(I+x)/I$ in $M/I$ correspond to $1$-dimensional spaces $x$ outside the space $L\oplus I:=X$ of codimension $1$ in $E$, and for all of them, $I+x$ has to be independent. We summarise this in a proposed new axiom (nI3).

\begin{Definition}\label{nuovoI3}
Let $E$ be a vector space and $\mathcal{I}$ a family of subspaces. We define the following property (axiom) of $\mathcal{I}$.
\begin{itemize}
\item[(nI3)] For all $I,J\in\kI$ satisfying $\dim I<\dim J$, there exists a codimension $1$ subspace $X\subseteq E$ with $I\subseteq X$, $J\not\subseteq X$ such that $I+x\in\kI$ for all $1$-dimensional $x\subseteq E$, $x\not\subseteq X$.
\end{itemize}
\end{Definition}

\begin{Remark}\label{vnI3impliesI3}
Because $J\not\subseteq X$, there is an $x\subseteq J$ such that $x\not\subseteq X$ and thus $I+x\in\kI$. This shows that (nI3) implies (I3). Also, (nI3) becomes (I3) in the classical case, since there is only one element $x$ outside $X$ that, by construction, is in $J$.
\end{Remark}

Looking back at the small example we started with, we see that letting $x\in\kI$ would imply, by applying (nI3) to $0$ and $x$, that at least on of $y$ and $z$ should also be in $\kI$.

\subsection{The independence axioms are equivalent to (I1), (I2), (nI3)}

In this section we prove that the axiom system (I1), (I2), (I3), (I4) is equivalent to the axiom system (I1), (I2), (nI3). First we show that the axioms (I1), (I2), (I3) and (I4), together, imply the new axiom (nI3).

\begin{Theorem}
Let $(E, \kI)$ be a $q$-matroid. Then, for the set $\kI$, the axiom (nI3) holds true.
\end{Theorem}
\begin{proof}
Let $I, J\in \kI$, $\dim(I)<\dim(J)$. Consider all $1$-spaces $y$ not in $I$ such that $I\in\max\{\kI\cap\kL(I+y)\}$. Let $A$ be the sum of all such $I+y$. We claim that $I\in\max\{\kI\cap\kL(A)\}$. This can be seen by applying (I4) multiple times. Let $y_1$ and $y_2$ be such that $I \in \max\{\kI\cap \kL(I+y_1)\}$ and $I\in\max\{\kI\cap \kL(I+y_2)\}$. Then, by (I4), $I\in\max\{\kI\cap\kL((I+y_1)+(I+y_2))\}$. Iterating this argument shows that $I\in\max\{\kI\cap\kL(A)\}$ and that, moreover, all $I+y$ with $y\subseteq A$, $y\not\subseteq I$ are not in $\kI$. \\
On the other hand, for all $1$-spaces $z\subseteq E$, $z\not\subseteq A$ we have that $I+z\in\kI$. We know from (I3) that there is at least one such $z$, namely the $x\in J$, $x\not\subseteq I$ such that $I+x\in\kI$. This means that $\dim A\leq\dim E-1=n-1$. Now take a space $X$ of codimension $1$ in $E$ with $A\subseteq X$ and $x\not\subseteq X$. Then for all $z\subseteq E$, $z\not\subseteq X$ we have that $I+z\in\kI$ and this proves (nI3).
\end{proof}

Next we show that the axiom (nI3), together with (I1) and (I2), implies the axioms (I3) and (I4). Before doing that, we prove a small lemma about restriction.

\begin{Lemma}\label{restrict}
Let $E$ be a vector space and let $ \tilde{\kI}$ be a collection of subspaces of $E$ satisfying the axioms (I1), (I2) and (nI3). Let $F\subseteq E$ and let $\tilde{\kI}|_F=\{I\in\tilde{\kI}:I\subseteq F\}$. Then $\tilde{\kI}|_F$ satisfies the axioms (I1), (I2) and (nI3).
\end{Lemma}
\begin{proof}
It is clear that $\tilde{\kI}|_F$ satisfies (I1) (because $\{0\}\in\tilde{\kI}$) and (I2). Let $I,J\in\tilde{\kI}|_F$ with $\dim I<\dim J$. Let $X$ be the codimension $1$ space in $E$ defined by axiom (nI3). Then $F\not\subseteq X$, because $J\not\subseteq X$ and $J\subseteq F$. Therefore, $X\cap F$ is a codimension $1$ space in $F$ that satisfies (nI3).
\end{proof}

\begin{Theorem}\label{verynewI3thm}
Let $E$ be a vector space and let $\tilde{\kI}$ be a collection of subspaces satisfying the axioms (I1), (I2) and (nI3). Then $(E,\tilde{\kI})$ is a $q$-matroid.
\end{Theorem}
\begin{proof}
We have to show that $\tilde{\kI}$ satisfies the axioms (I1), (I2), (I3) and (I4). The first two axioms are satisfied by definition and (nI3) implies (I3) as was noted in Remark \ref{vnI3impliesI3}, so it is left to prove (I4). By Theorem \ref{oldI4s} it is enough to prove the axiom (I4''):
\begin{itemize}
\item[(I4'')] Let $A \in \kL(E)$ and let $I \in \max(A,\I)$. 
Let $x \in \kL(E)$ be a $1$-dimensional space. Then there exists $J \in \max(x+A,\I)$ such that $J \subseteq x+I$.
\end{itemize}
Thanks to Lemma \ref{restrict} we can let $n=\dim(A+x)$ and restrict to $A+x$. (I4'') is direct if $x\subseteq A$ or if $I\in\max(\tilde{\kI},A+x)$, so suppose both are not the case. Then there is a $J\in\max(\tilde{\kI},A+x)$ with $\dim J>\dim I$. Moreover, $J\not\subseteq A$ because that would contradict the maximality of $I$. By (nI3), there is a codimension $1$ space $X$ in $A+x$ such that $I\subseteq X$, $J\not\subseteq X$ and for all $y\not\subseteq X$ we have $I+y\in\tilde{\kI}$. We now claim that $X=A$. If not, there would be a $y\not\subseteq X$, $y\subseteq A$ such that $I+y\in\tilde{\kI}$. Since $I+y\subseteq A$, this contradicts the maximality of $I$. So, $X=A$ and by (nI3) we have that $I+x\in\tilde{\kI}$. Moreover, $I+x\in\max(A+x,\tilde{\kI})$ because if there is a member of $\tilde{\kI}$ of bigger dimension in $A+x$, its intersection with $A$ would have dimension strictly bigger then $\dim I$, which contradicts, again, $I\in\max(A,\tilde{\kI})$. That proves (I4'') and shows that $(E,\tilde{\kI})$ is a $q$-matroid.
\end{proof}

\section{A new bases axiom}

As a consequence of our introduction of (nI3), we could define a new basis axiom (nB3), which again avoids the presence of the fourth axiom (B4).

\begin{Definition}
Let $E$ be a vector space and $\mathcal{B}$ a family of subspaces. We define the following property (axiom) of $\mathcal{B}$.
\begin{itemize}
\item[(nB3)] For all $B_1,B_2\in\kB$, and for each subspace $A$ that has codimension $1$ in $B_1$ \textcolor{red}{\sout{containing $B_1 \cap B_2$}} there exists $X\subseteq E$ of codimension $1$ in $E$ such that $X \supseteq A$, $X\not \supseteq B_2$ and $A+x \in \mathcal{B}$ for all $1$-dimensional $x\subseteq E$, $x\not\subseteq X$.
\end{itemize}
\end{Definition}

\begin{Remark}\label{basisclassical2}
\textcolor{red}{Note that contrarily to (B3), it is possible that $B_1=B_2$. Also, we drop the requirement that $B_1\cap B_2\subseteq A$. This requirement was needed in (B3) to make sure that $A+x$ was different from $B_1$. In (nB3), many bases are produced: some of them might be equal to $B_1$, but some are new. (Unless $B_1=E$, but then there is only one basis anyway.)}
Because $B_2\not\subseteq X$, there is an $x\subseteq B_2$ such that $x\not\subseteq X$ and thus $A+x\in\kB$. This shows that (nB3) implies (B3). Also, (nB3) becomes (B3) in the classical case, since there is only one element $x$ outside $X$ that, by construction, is in $B_2$.
\end{Remark}

To show that (nB3) holds, we use a similar approach to what was done for independent spaces. We prove that the axiom system (B1), (B2), (B4) is equivalent to the axiom system (B1), (B2), (nB3). Since we showed in Theorem \ref{B3isImplByB4} that (B4) implies (B3), we can freely use (B3) within the proofs for convenience. First we show that the axioms (B1), (B2) and (B4), together, imply the new axiom (nB3).

\begin{Theorem}\label{nB3holds}
Let $(E, \kB)$ be a $q$-matroid. Then, for the set $\kB$, the axiom (nB3) holds true.
\end{Theorem}
\begin{proof}
Let $B_1, B_2\in \kB$ and let $A\subseteq B_1$ of codimension $1$ \textcolor{red}{\sout{such that $B_1\cap B_2\subseteq A$}}. Consider all $1$-spaces $y$ not in $A$ such that $A \in \max(E,(A+y)\cap \kB)$. Let $C$ be the sum of \textcolor{red}{\sout{all such $A+y$}} \textcolor{red}{$A$ and all such $y$}. We claim that $A \in \max(E,C\cap \kB)$.
This can be seen by applying (B4) multiple times. Since $A\in\max(E,(A+y_1)\cap\kB)$ and $A\in\max(E,(A+y_2)\cap\kB)$, by (B4) $A\in\max(E,((A+y_1)+(A+y_2))\cap\kB)$. Iterating this argument shows that $A\in\max(E,C\cap\kB)$ and moreover, all $A+y$ with $y\subseteq C$, $y\not\subseteq A$ are not in $\kB$. \\
On the other hand, for all $1$-spaces $z\subseteq E$, $z\not\subseteq C$ we have that $A+z\in\kB$. \textcolor{red}{\sout{We know from (B3) that there is at least one such $z$, namely the $x\in B_2$, $x\not\subseteq B_1$ such that $A+x\in\kB$.}} \textcolor{red}{We know such $z$ exists, for example any $z$ with $z\subseteq B_1$, $z\not\subseteq A$.} This means that $\dim C\leq\dim E-1=n-1$. Now take a space $X$ of codimension $1$ in $E$ with $C\subseteq X$ and $x\not\subseteq X$. Then for all $z\subseteq E$, $z\not\subseteq X$ we have that $A+z\in\kB$ and this proves (nB3).
\end{proof}

Now we work towards the converse of this theorem. For this we will use the following two variations of the axiom (B4).

\begin{itemize}
    \item[(B4')] Let $A,B\subseteq E$ and $I\in\max(E,A\cap\kB)$. Then there exists $J\in\max(E,(A+B)\cap\kB)$ such that $J\subseteq I+B$. 
    \item[(B4'')]  Let $A\subseteq E$ and $I\in\max(E,A\cap\kB)$. Let $x \subseteq E$ be a one-dimensional space. Then, there exists $J\in\max(E,(A+x)\cap\kB)$ such that $J\subseteq x+I$.
\end{itemize}

The next result shows that we can in fact take these axioms to define a $q$-matroid. It is the statement of Proposition \ref{newI4s} but for bases instead of independent spaces, and the proof is similar to to proof for independence axioms in \cite[Theorem 26]{bcj}.

\begin{Theorem}\label{B4Variations}
Let $\kB$ be a collection of subspaces satisfying (B1) and (B2).
Then the axiom systems (B1), (B2), (B4); (B1), (B2), (B4') and (B1), (B2), (B4'') are pairwise equivalent.
\end{Theorem}

\begin{proof}
First, we assume that (B4) holds for the collection $\kB$ and we show that this implies that (B4') and (B4'') also hold.
Let $A,B\subseteq E$ and let $I\in\max(E,A\cap\kB)$ and $J\in\max(E,B\cap\kB)$. By (B4) there is a $K\in\max(E,(A+B)\cap\kB)$ with $K \subseteq I+J$. Since $J\subseteq B$, $K$ is also contained in $I+B$. This shows (B4'). We get (B4'') by taking $B=x$.

Suppose that (B4'') holds. We will show that (B4') holds. Let $A ,B\subseteq E$ and let $I\in\max(E,A\cap\kB)$.
Suppose that (B4') holds for all subspaces of dimension less than $\dim(B)$. Let $C$ be a subspace of $B$ of codimension 1 in $B$ and write $B=x+C$. By hypothesis, there exists $J\in\max(E,(A+C)\cap\kB)$ such that $J \subseteq I+C$. By (B4'') there exists 
$J'\in\max(E,(A+C+x)\cap\kB) = \max(E,(A+B)\cap\kB)$  such that $J' \subseteq J + x \subseteq I+C+x=I+B$. This proves (B4').
    
Now suppose that (B4') holds. Let $A,B\subseteq E$ and let 
$I\in\max(E,A\cap\kB)$ and $J\in\max(E,B\cap\kB)$. We claim there is a $K\in\max(E,(A+B)\cap\kB)$ with $K \subseteq I+J$.
Since  $J\in\max(E,B\cap\kB)$, applying (B4') to $B$ and $I$ gives that these exists $N\in\max(E,(I+B)\cap\kB)$ such that $N \subseteq I+J$.
  
Again by (B4'), there exists 
$M\in\max(E,(A+B)\cap\kB)$ such that $M \subseteq I+B$. But then $M\in\max(E,(I+B)$ and hence $M$ and $N$ have the same dimension.
It follows that  $N\in\max(E,(A+B)\cap\kB)$ and $N \subseteq I+J$ and so (B4') implies (B4). The result follows.
\end{proof} 

We now prove the converse of Theorem \ref{nB3holds}.

\begin{Theorem}\label{verynewB3thm}
Let $E$ be a vector space and let $\tilde{\kB}$ be a collection of subspaces satisfying the axioms (B1), (B2) and (nB3). Then $(E,\tilde{\kB})$ is a $q$-matroid.
\end{Theorem}
\begin{proof}
We have to show that $\tilde{\kB}$ satisfies the axioms (B1), (B2), and (B4), since then (B3) is implied by Theorem \ref{B3isImplByB4}. The first two axioms are satisfied by definition, so it is left to prove (B4). By Theorem \ref{B4Variations}, it is enough to prove the axiom (B4'').  

Let $A\subseteq E$ and 
$I\in\max(E,A\cap \tilde{\kB})$. Let $x\subseteq E$ be a $1$-dimensional space. (B4'') is direct if $x\subseteq A$ or if
$I\in\max(E,(A+x)\cap \tilde{\kB})$, so suppose both are not the case. Then there is a $J\in\max(E,(A+x)\cap \tilde{\kB})$ with $\dim J>\dim I$. Moreover, $J\not\subseteq A$ because that would contradict the maximality of $I$. 

Since $I$ and $J$ are intersections of members of $\tilde{\kB}$ with $A$ and $A+x$, respectively, we can find $B_1,B_2\in\tilde{\kB}$ such that $I=B_1\cap A$ and $J=B_2\cap(A+x)$. Moreover, there is a codimension $1$ subspace $C\subseteq B_1$ such that $C\cap(A+x)=I$ \textcolor{red}{\sout{and $B_1\cap B_2\subseteq C$}}. Now we apply (nB3) to $B_1,B_2$ and $C$. This gives a codimension $1$ space $X\subseteq E$ such that $C\subseteq X$, $B_2\not\subseteq X$ and $C+y\in \tilde{\kB}$ for all $1$-dimensional $y\subseteq E$, $y\not\subseteq X$.

We now claim that $X\cap(A+x)=A$. If not, there would be a $z\not\subseteq X$, $z\subseteq A$ such that $C+z\in\tilde{\kB}$. Then $(C+z)\cap(A+x)=I+z$ and this contradicts the maximality of $I$ in $A$. So, $X\cap(A+x)=A$ and in particular $x\not\subseteq X$, so by (nB3) we have that $C+x\in\tilde{\kB}$. Moreover, $I+x\in\max(E,(A+x)\cap \tilde{\kB})$ because if there is a bigger intersection with a member of $\tilde{\kB}$ in $A+x$, its intersection with $A$ would have dimension strictly bigger then $\dim I$, which contradicts, again, that $I\in\max(E,A\cap \tilde{\kB})$. That proves (B4'') and shows that $(E,\tilde{\kB})$ is a $q$-matroid.
\end{proof}

\section{A new spanning spaces axiom}

In this section we state and prove that it is possible to define spanning spaces with three axioms, this being an easy consequence of what we did for independent spaces.

\begin{Definition}
Let $E$ be a vector space and $\mathcal{S}$ a family of subspaces. We define the following property (axiom) of $\mathcal{S}$.
\begin{itemize}
\item[(nS3)] For all $S_1,S_2\in\mathcal{S}$ satisfying $\dim S_2<\dim S_1$, there exists a $1$-dimension subspace $x\subseteq S_1$, $x\not\subseteq S_2$ such that for all codimension-one $X\subseteq E$ with $X\not\supseteq x$ we have $X\cap S_1\in\kS$.
\end{itemize}
\end{Definition}

\begin{Theorem}
Let $E$ be a vector space and let $\kS$ be a family of subspaces satisfying (S1), (S2) and (nS3). Then $\kS$ is the family of spanning spaces of a $q$-matroid.
\end{Theorem}
\begin{proof}
This follows directly from the fact that spanning spaces are the dual spaces of the independent spaces of the dual $q$-matroid, and the axiom (nS3) is the dual statement of (nI3).
\end{proof}

\section{Two new \emph{q}-cryptomorphisms}

In this section we apply our results on new axioms for independent spaces and bases to derive two new cryptomorphisms: between circuits and independent spaces, and between independent spaces and bases. The latter was already done in \cite{JP18} but as we will discuss, we believe there is a gap in that proof.

We will see that while (nI3) might feel like a natural $q$-analogue of the third axiom for matroids, it turns out that (oI4) or (I4'') is much more practical in proofs.

\subsection{Circuits and independent spaces}
Here we prove that the axioms (C1), (C2) and (C3) are equivalent to the axioms (I1), (I2) and (I4''). We follow Lemma 1.1.3 and Theorem 1.1.4 of \cite{oxley}.

\begin{Theorem}\label{FromItoC}
Let $(E,\mathcal{I})$ be a $q$-matroid. Define
\[ \mathcal{C}_\mathcal{I}=\{C\subseteq E:C\notin\mathcal{I},I\in\mathcal{I}\text{ for all } I\subsetneq C\}. \] 
Then $\mathcal{C}$ is a family of circuits, that is, it satisfies the axioms (C1), (C2) and (C3).
\end{Theorem}
\begin{proof}
The axioms (C1) and (C2) follow directly from the definition of $\mathcal{C}$. We will prove (C3) by making use of the independence axiom (oI4).

Let $C_1,C_2 \in \kC$, $C_1\neq C_2$ and $X \subseteq E$ a codimension one space.
Suppose towards a contradiction $(C_1+C_2)\cap X $ does not contain any circuit, this making it an independent space.
Let $I_1 \subseteq C_1$ and $I_2 \subseteq C_2$ be of codimension $1$ such that $C_1\cap C_2=I_1\cap I_2$. Note that $I_1$ and $I_2$ are independent and since they have codimension $1$ they are maximal with respect to both dimension and inclusion.
By (oI4) there is an $I\in\max\{C_1+C_2,\kI\}$ such that $I\subseteq I_1+I_2$.
Let $F\subseteq E$ a codimension $1$ space containing $C_1+I_2$ but not containing $C_2$ and $G\subseteq E$ a codimension $1$ space containing $C_2+I_1$ but not containing $C_1$. Clearly $F \neq G$, so $\dim((C_1+C_2)\cap F \cap G)= \dim(C_1+C_2)-2$.
Now, $I \subseteq ((C_1+C_2)\cap F \cap G)$, so 
\[\dim(I) \leq \dim(C_1+C_2)-2 < \dim((C_1+C_2)\cap X). \]
However, by assumption $(C_1+C_2)\cap X$ is an independent space, this giving a contradiction with the maximality of $I$.
We conclude that (C3) needs to hold.
\end{proof}

\begin{Theorem}\label{FromCtoI}
Let $(E,\mathcal{C})$ be a $q$-matroid. Define
\[ \mathcal{I}_\mathcal{C}=\{I\subseteq E:C\not\subseteq I\text{ for all }C\in\mathcal{C}\}. \]
Then $\mathcal{I}$ is a family of independent spaces, that is, it satisfies the axioms (I1), (I2) and (I4'').
\end{Theorem}
\begin{proof}
The axioms (I1) and (I2) follow directly from the definition of $\mathcal{I}$. For (I4''), let $A\subseteq E$ and let $I\subseteq A$ be a maximal independent subspace. (Throughout this proof, maximality is always taken with respect to dimension.) Let $x\subseteq E $ be a $1$-dimensional space. If $x\subseteq A$, (I4'') clearly holds. If $I=A$, we also get that (I4'') holds: either $I$ is a maximal independent space in $A+x$, or $I+x=A+x$ is independent itself. So assume $x\not\subseteq A$ and $I\neq A$.

Towards a contradiction, suppose (I4'') does not hold for $A$ and $x$. Let $J$ be a maximal independent subspace in $A+x$. If $\dim J=\dim I$, we have that $I$ is also a maximal independent space in $A+x$, contradicting that there are no maximal independent spaces in $I+x$. So we have that $\dim J>\dim I$. In fact, $\dim J=\dim I+1$, because otherwise $J\cap A$, that is independent and has dimension at least $\dim J-1$, contradicts the maximality of $I$ in $A$.

There might be several choices for $J$: pick one such that $I\cap J$ is maximal. We claim that $I$ cannot be contained in $J$. If this was the case, we can write $J=I+y\in\mathcal{I}$ but $I+x\notin\mathcal{I}$ by construction. This implies $x\not\subseteq J$ hence $J+x\notin \mathcal{I}$. Thus $I+x$ and $J+x$ both contain a circuit. Apply (C3) to these circuits with a codimension $1$ space $Y$ such that $A\subseteq Y$, $A+y\not\subseteq Y$. This yields a circuits inside $J\cap A$, which is a contradiction because $J\cap A$ is independent by (I2). So, $I\not\subseteq J$.

We pick a codimension $1$ space $X\subseteq E$ and a $1$-dimensional space $e$ such that $J\subseteq X$, $I\not\subseteq X$, $e\subseteq I$ and $e\not\subseteq X$. This implies that $(J+e)\cap X=J$. For any codimension $1$ space $F\subseteq E$ with $I\subseteq F$, $J\not\subseteq F$ we can now construct the following. Define $T_F=(J+e)\cap F$. Since $e\not\subseteq J$ and $J\not\subseteq F$ we have that $\dim T_F=\dim J$. However, $\dim T_F\cap I>\dim J\cap I$, so by assumption on our choice of $J$ we have that $T_F$ is not independent, hence contains a circuit $C_F$. We cannot have that $C_F\subseteq X$, because that would imply $C_F\subseteq (J+e)\cap F\cap X = J\cap F\subseteq J$ and the latter is independent.

Let $G,H\subseteq E$ be two distinct codimension $1$ spaces with $I\subseteq G,H$ and $J\not\subseteq G,H$. These exist, because $\dim J-\dim(J\cap I) \geq \dim J-(\dim I-1)=2$. We can also assume that $J-G\neq J-H$. Now we apply the construction as above to obtain $C_G$ and $C_H$. From our last assumption it follows that $C_G\neq C_H$. Now apply (C3) to $C_G, C_H$ and $X$. This gives a circuit $C\subseteq (C_G+C_H)\cap X$. Note that since $C_G,C_H\not \subseteq X$, $C\neq C_G,C_H$. Now $C_G+C_H\subseteq (J+e)\cap G + (J+e)\cap H\subseteq J+e$ so $C \subseteq (C_G+C_H)\cap X\subseteq (J+e)\cap X=J$. This is a contradiction because $J$ is independent. We conclude that (I4'') needs to hold.
\end{proof}

\begin{Corollary}\label{WeHaveCrypt}
Let $(E,\kI)$ be a collection of independent spaces and let $(E,\kC)$ be a collection of circuits. 
\begin{enumerate}
\item $(E,\kI)$ determines a $q$-matroid with collection of independent spaces $\kI$ and collection of circuits $\kC_\kI$.
\item $(E,\kC)$ determines a $q$-matroid with collection of circuits $\kC$ and collection of independent spaces  $\kI_\kC$.
\end{enumerate}
\end{Corollary}
\begin{proof}
It was shown in \cite{bcj} that $(E,\mathcal{I})$ determines a $q$-matroid with collection of independent space $\mathcal{I}$ and that $(E,\mathcal{C})$ determines a $q$-matroid with collection of circuits $\mathcal{C}$.
The statements now follow from Theorems \ref{FromItoC} and \ref{FromCtoI} and the straightforward result that $\mathcal{I}_{\mathcal{C}_\mathcal{I}}=\mathcal{I}$ and $\mathcal{C}_{\mathcal{I}_\mathcal{C}}=\mathcal{C}$.
\end{proof}

\subsection{Bases and independent spaces}

A cryptomorphism between independent spaces and bases was proven in \cite[Theorem 37]{JP18}. However, we believe that there is a gap in that proof. In \cite[Theorem 37]{JP18}, one of the steps is assuming a collection of bases, defining $\kI_\kB=\{I\subseteq B:B\in\kB\}$, and proving $\kI$ satisfies the axioms (I1)-(I4). In the proof of (I3), truncation is used. It was proven earlier in \cite{JP18} that the truncation of a $q$-matroid, defined by its rank function, is again a $q$-matroid. However, when assuming a collection of bases, this result is not valid: it only becomes valid once a cryptomorphism between bases and the rank function is established. This is not yet the case -- in fact, a cryptomorphism between bases and the rank function would be a corollary of the cryptomorphism between bases and independent spaces, which is the goal of \cite[Theorem 37]{JP18}.

In order to fix this issue, we can use our results from Section \ref{redundancy} that show the redundancy of the axioms (I3) and (B3). As was also mention in \cite{JP18}, the axioms (I4) and (B4) are easily related to each other. The next lemma makes this precise.

\begin{Lemma}\label{I4B4TheSame}
\begin{enumerate}
\item Let $\kB$ be a collection of 
subspaces of $E$ satisfying (B1) and (B2). Define $\kI_\kB=\{I\subseteq B:B\in\kB\}$. Then for all $A\subseteq E$, $\max(E,A \cap \kB)= \max(A,\kI_\kB)$.
\item Let $\kI$ be a collection of subspaces of $E$ satisfying (I1) and (I2). Define $\kB_\kI=\max(\kI)$. Then for all $A\subseteq E$, $\max(A,\kI)=\max(E,A\cap\kB_\kI)$.
\end{enumerate}
\end{Lemma}
\begin{proof}
\begin{enumerate}
\item Let $I\in\max(E,A \cap \kB)$. Then by definition, $I\in\kI_\kB$. Suppose there is an $I'\in\max(A,\kI_\kB)$ with $\dim(I)< \dim(I')$. Then there would be a $B\in\kB$ such that $I'\subseteq B$, hence $I'\in\max(E, A\cap \kB)$, contradicting the maximality of $I$.
For the reverse inclusion, let $I\in\max(A,\kI_\kB)$. Then there is a $B\in\kB$ such that $I=B\cap A$. Suppose there is a $B'\in\kB$ such that $\dim(B' \cap A) > \dim(I)$. Then $B' \cap A\in \kI_\kB$, contradicting the maximality of $I$. This proves that $\max(E,A \cap \kB)= \max(A,\kI_\kB)$.
\item It was proven in \cite[Theorem 37]{JP18} that $\kB_\kI$ satisfies the axioms (B1) and (B2). Also, it was shown that $\kB_{\kI_\kB}=\kB$ and that $\kI_{\kB_\kI}=\kI$. Applying the first part of this lemma to $\kB_\kI$ gives that $\max(E,A \cap \kB_\kI)= \max(A,\kI_{\kB_\kI})=\max(A,\kI)$. \qedhere
\end{enumerate}
\end{proof}

\begin{Corollary}\label{WeHaveCrypt2}
Let $(E,\kI)$ be a collection of independent spaces and let $(E,\kB)$ be a collection of bases. 
\begin{enumerate}
\item $(E,\kI)$ determines a $q$-matroid with collection of independent spaces $\kI$ and collection of bases $\kB_\kI$.
\item $(E,\kB)$ determines a $q$-matroid with collection of bases $\kB$ and collection of independent spaces $\kI_\kB$.
\end{enumerate}
\end{Corollary}
\begin{proof}
By Theorem \ref{I3isImplByI4}, a collection of independent spaces is completely determined by the axioms (I1), (I2) and (oI4) and moreover, $(E,\kI)$ defines a $q$-matroid. By Theorems \ref{B3isImplByB4} and \ref{B4Variations}, a collection of bases is completely determined by the axioms (B1), (B2) and (B4).

Assume $\kI$ satisfies (I1), (I2), (I4). By \cite[Theorem 37]{JP18}, $\kB_\kI$ satisfies (B1) and (B2) and by Lemma \ref{I4B4TheSame} it satisfies (B4). For the converse, assume $\kB$ satisfies (B1), (B2) and (B4). Again by \cite{JP18}, $\kI_\kB$ satisfies (I1) and (I2) and by Lemma \ref{I4B4TheSame} it satisfies (oI4). Finally, $\kB_{\kI_\kB}=\kB$  and $\kI_{\kB_\kI}=\kI$ is also proven in \cite[Theorem 37]{JP18}.
\end{proof}

\subsection*{Acknowledgement}
The authors wish to thank Heide Gluesing-Luerssen for her comments leading to the discussion in Section \ref{discussionI4}, and also the attendants of the special session ``$q$-Analogues in combinatorics: matroids, designs and codes'' at Applications of Computer Algebra (ACA 2022) for their fruitful discussions. Furthermore, we thank the reviewer for their suggestions that improved the exposition of this paper.
\smallskip

M. Ceria has been supported by the project ``Schemi crittografici e di trasmissione dei dati per database distribuiti e in cloud”
in the program
``Research for Innovation'' (REFIN) - POR Puglia FESR FSE 2014-2020 Codice CUP: D94I20001410008.\\
She has been partially supported by the ``National
Group for Algebraic and Geometric Structures, and their Applications'' (GNSAGA - INdAM).

\bibliographystyle{abbrv}
\bibliography{biblio} 

\end{document}